\documentclass[11pt,letterpaper]{amsart}
\usepackage{preamble}
\counterwithin{equation}{section}

\title{Diameter Bounds for Friends-and-Strangers Graphs}
\author[Amogh Akella]{Amogh Akella}
\address[]{Westwood High School, Austin, TX 78750, USA}
\email{hgomamogh@gmail.com}
\author[Rupert Li]{Rupert Li}
\address[]{Stanford University, Stanford, CA 94305, USA}
\email{rupertli@stanford.edu}
\date{\today}

\begin{document}

\begin{abstract}
Consider two $n$-vertex graphs $X$ and $Y$, where we interpret $X$ as a social network with edges representing friendships and $Y$ as a movement graph with edges representing adjacent positions.
The friends-and-strangers graph $\mathsf{FS}(X,Y)$ is a graph on the $n!$ permutations $V(X)\to V(Y)$, where two configurations are adjacent if and only if one can be obtained from the other by swapping two friends located on adjacent positions.
Friends-and-strangers graphs were first introduced by Defant and Kravitz, and generalize sliding puzzles as well as token swapping problems. Previous work has largely focused on their connectivity properties.

In this paper, we study the diameter of the connected components of $\FS(X, Y)$. We extend the result of Kornhauser, Miller, and Spirakis on sliding puzzles to general graphs in two ways. First, we show that the diameter of $\FS(X, Y)$ is polynomially bounded when both the friendship and the movement graphs have large minimum degree. Second, when both the underlying graphs $X$ and $Y$ are Erd\H os-R\' enyi random graphs, we show that the distance between any pair of configurations is almost always polynomially bounded under certain conditions on the edge probabilities.
\end{abstract}

\maketitle

\section{Introduction}

Let $X$ be a ``friendship graph'' on $n$ vertices, where the vertices represent people and the edges represent friendships. Let $Y$ be a ``movement graph'' on $n$ locations. Suppose that the $n$ people in $X$ are placed on the vertex set of $Y$, each in a distinct location. Two people can swap positions in a ``friendly swap'' if they are friends in $X$ and are in adjacent positions in $Y$.
The following two problems naturally come to mind.
\begin{question}\label{question:oneone} Given two configurations of people on the movement graph, can one be obtained from the other via a sequence of friendly swaps?
\end{question}
\begin{question}\label{question:onetwo}
What is the maximum number of moves needed to go from one configuration to another?
\end{question}

More formally, let $G$ be a graph on $n!$ vertices. 
We say that $G$ is the \emph{friends-and-strangers graph} of $X$ and $Y$, denoted as $\FS(X, Y)$, if and only if it satisfies the following conditions: 

\begin{enumerate}
    \item there exists a bijection $f$ from the vertex set of $G$ to the set of permutations mapping the people in $X$ onto $Y$;
    \item let $v$ and $w$ be any two vertices in $G$. Then, $v$ and $w$ are adjacent in $G$ if and only if $f(v)$ is reachable from $f(w)$ in exactly one friendly swap. 
\end{enumerate}

Friends-and-strangers graphs are an extension of the famous 15-puzzle, which involves sliding tiles labeled 1 through 15 along with an empty slot in a 4-by-4 grid until a specific configuration is reached. 
Indeed, it can be easily verified that the graph of configurations of the 15-puzzle, with edges corresponding to possible moves, is isomorphic to the graph $\FS(\Star_{16},\Grid_4)$.\footnote{$\Star_n$ is the star graph with $n$ vertices and $\Grid_n$ is an $n\times n$ grid graph. See \Cref{sec:background} for formal definitions.} 

More generally, sliding puzzles can be defined on arbitrary $n$-vertex graphs $Y$ by placing $n-1$ tiles on distinct vertices, with one vertex left unoccupied. A single move consists of moving a tile from its current location to an adjacent unoccupied location. The movement graph over configurations of the tiles is then given by $\FS(\Star_n, Y)$ where the tiles correspond to the leaves of the star graph and the center of the star corresponds to the ``hole'' representing the unoccupied position. Solving the puzzle corresponds to reaching a target configuration through a sequence of legal moves. Connectivity in $\FS(\Star_n, Y)$, i.e., \cref{question:oneone}, is akin to asking whether the puzzle is solvable from a given starting configuration. \cref{question:onetwo}, on the other hand, asks how quickly the puzzle can be solved. In other words, given that the target configuration is reachable, how many steps are needed to reach it?

Connectivity in $\FS(\Star_n, Y)$ has been extensively studied starting with the work of Wilson~\cite{wilson1974graph} and is well understood. Additionally, general bounds for the diameter of connected components in this graph were proven by Kornhauser, Miller, and Spirakis \cite{715921}. The main contribution of this work is to use the latter result to prove diameter bounds for other families of friends-and-strangers graphs.

Before we elaborate on our results, we describe previous work on friends-and-strangers graphs and related problems.

\subsubsection*{Connectivity in friends-and-strangers graphs.} Friends-and-strangers graphs were first introduced in their full generality and studied by Defant and Kravitz~\cite{defant2021friends}. They studied connectivity for special cases such as  $\FS(\Path_n, Y)$ and $\FS(\Cycle_n, Y)$, and derived necessary and sufficient conditions for $\FS(X,Y)$ to be connected. Subsequently, connectivity has been extensively studied for other families of friends-and-strangers graphs \cite{alon2021extremal,bangachev2022asymmetric,defant2022connectedness,defant2021friends,milojevic2022connectivity,Stanley2012equivalence,wang2023connectivity}. 
Stronger notions of connectivity, namely biconnectivity and $k$-connectivity, have also been studied~\cite{krishnan2024connectivity}. 

Alon, Defant and Kravitz \cite{alon2021extremal} initiated the study of extremal properties of friends-and-strangers graphs, including connectivity as a function of the minimum degrees of the graphs $X$ and $Y$ as well as when both $X$ and $Y$ are Erd\H os-R\'enyi random graphs. 
Bangachev \cite{bangachev2022asymmetric} followed up with an almost complete characterization of connectivity as a function of the minimum degrees of the graphs $X$ and $Y$. Wang and Chen \cite{wang2023connectivity} and Milojevi\'c \cite{milojevic2022connectivity} further refined the connectivity results for random graphs.

\subsubsection*{The diameter of friends-and-strangers graphs.} The focus of our work is on \cref{question:onetwo}, namely the diameter of friends-and-strangers graphs. Defant and Kravitz first proposed this as an open problem in their work \protect{\cite[Section 7.3]{defant2021friends}}. They noted that although in many settings $\FS(X,Y)$ is not connected, one may ask whether the diameter of each connected component in the graph is $\poly(n)$. (Observe that the size of $\FS(X,Y)$ is $n!$ -- superexponential in $n$.) In other words, starting from any configuration, is it possible to get to any {\em reachable} configuration in a small number of steps?

Jeong~\cite{jeong2022diameters} provided some partial answers to that question. He showed the existence of a family of friends-and-strangers graphs that contain connected components of diameter $e^{\Theta(n)}$. Jeong also investigated the diameter of some special classes of graphs; these results are reported in Table~\ref{table:our-results}. 

\subsubsection*{Token swapping.} Friends-and-strangers graphs are also a generalization of the {\em token swapping} problem that has been studied extensively in theoretical
computer science. In this problem, $n$ distinct tokens are places on the vertices of an $n$-vertex graph $Y$. A move consists of swapping any two tokens whose locations are adjacent in $Y$. The graph of configurations corresponding to token swapping is then precisely given by $\FS(K_n, Y)$, where $K_n$ is the complete graph on $n$ vertices.
Yamanaka, Demaine, Ito, Kawahara, Kiyomi, Okamoto, Saitoh, Suzuki, Uchizawa, and Uno \cite{Yamanaka2015SwappingLT} first introduced the token swapping problem, and studied the complexity of finding the shortest path between any pair of configurations. They noted that if $Y$ is connected, then so is $\FS(K_n, Y)$, and its diameter is always bounded by $O(n^2)$. They accordingly asked whether the shortest path between any pair of configurations can be found in time polynomial in $n$.
Miltzow, Narins, Okamoto, Rote, Thomas, and Uno \cite{miltzow2016approximation} showed that this problem is NP-hard, and subsequent work has consequently studied the approximability of the problem \cite{DBLP:journals/corr/BonnetMR16}.

\subsection{Our Results}
In this paper we provide new diameter bounds for a number of interesting families of friends-and-strangers graphs. 

Our main contribution is establishing diameter bounds for families of graphs $X,Y$ that have large minimum degree. This investigation is inspired by the work of Bangachev~\cite{bangachev2022asymmetric} who provided similar conditions for the connectivity of $\FS(X,Y)$.

\begin{restatable}{theorem}{mindegreemain}\label{theorem:mindegreemain}
Let $X$ and $Y$ be connected graphs on $n$ vertices such that 
\[ \min(\delta(X), \delta(Y)) + 2\max(\delta(X), \delta(Y)) \geq 2n.\]
Then $\FS(X, Y)$ is connected and has diameter at most $O(n^5)$. 
\end{restatable}

\begin{restatable}{theorem}{mindegreesecond}\label{theorem:mindegreesecond}
Let $X$ and $Y$ be connected graphs on $n$ vertices such that $\delta(X) + \delta(Y) \geq \frac{3n}{2}$.
Then $\FS(X, Y)$ is connected and has diameter at most $3n(n-1)/2$. 
\end{restatable}

We can visualize these results, along with the results of Bangachev \cite{bangachev2022asymmetric}, in \cref{fig:Banga}, where we identify pairs of values of $\delta(X)$ and $\delta(Y)$ for which $\FS(X,Y)$ is always connected, sometimes disconnected, or always has polynomial diameter. Observe that we prove a polynomial bound on the diameter for most of the regions where Bangachev~\cite{bangachev2022asymmetric} proved connectivity. The only exception is the green dotted region. This leads us to conjecture that whenever $\FS(X,Y)$ is connected, it has polynomially bounded diameter. Resolving this conjecture is the main problem left open by our work. See \cref{sec:openquestions} for a discussion of open questions. 

\begin{figure}[h]
    \centering
    \includegraphics[width=0.7\textwidth]{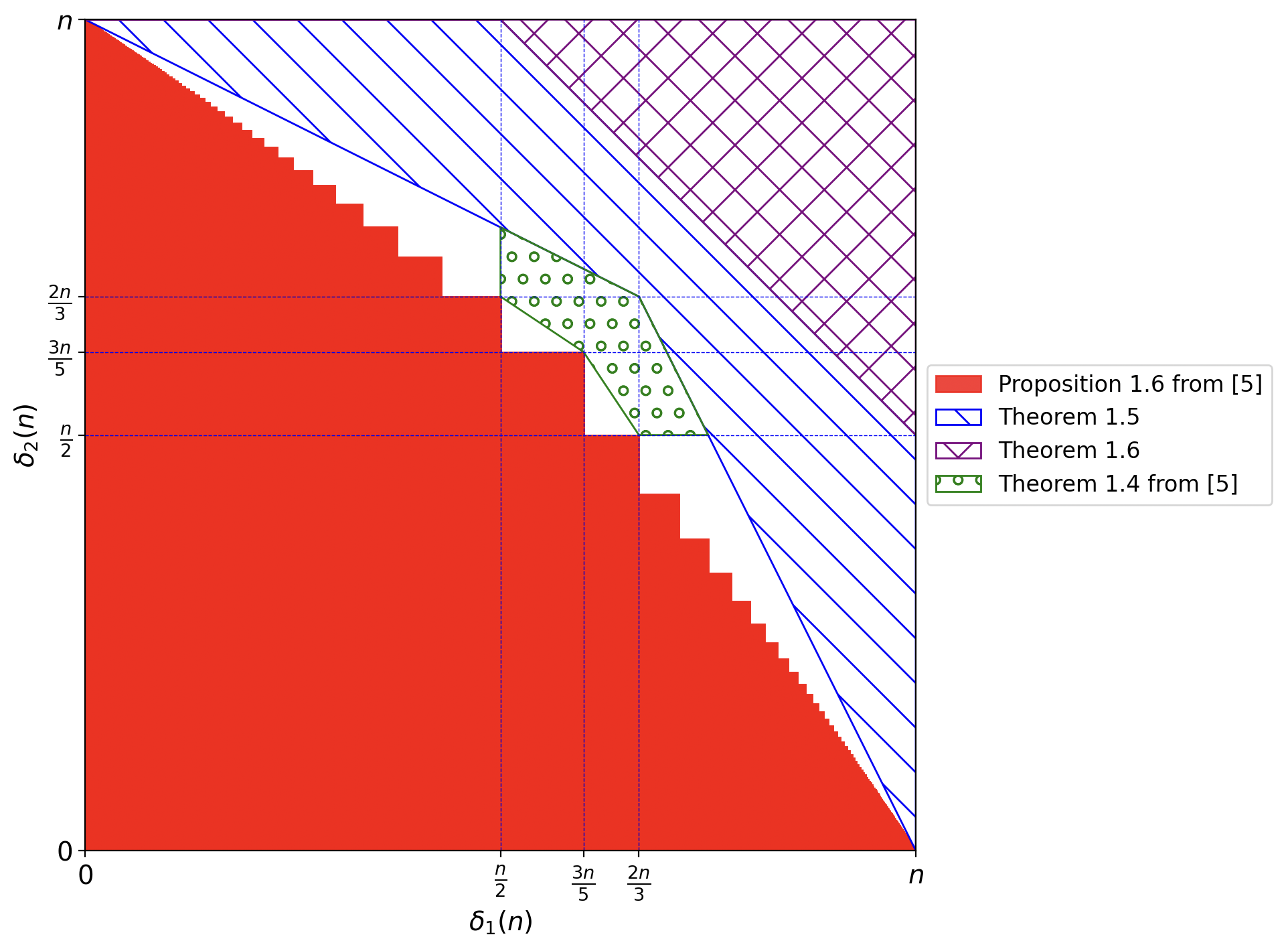}
    \caption{Points $(\delta_1(n), \delta_2(n))$ for which $\FS(X, Y)$ must be connected (green), must be connected with polynomial diameter (blue, purple), or can be disconnected (red). Additive constants are omitted. }
    \label{fig:Banga}
\end{figure}

Our next result studies the diameter of $\FS(X,Y)$ when both $X$ and $Y$ are Erd\H os-R\'enyi random graphs.
Our investigation is inspired by the work of Alon, Defant, and Kravitz~\cite{alon2021extremal}, who established sharp thresholds for the probability $p$ to ensure the connectivity of $\FS(X,Y)$ when $X$ and $Y$ are both independently drawn from $\mathcal{G}(n,p)$. We focus on a single pair of permutations over the vertex set of $X$ and determine a condition on the edge probabilities $p$ and $q$ such that the pair is connected by a short path in $\FS(X,Y)$ when the graphs are drawn from $\mathcal{G}(n,p)$ and $\mathcal{G}(n,q)$ respectively. This result builds upon the $\FS(\Star_n, Y)$ result in \cite{715921}. 

\begin{restatable}{theorem}{randommain}
\label{theorem:randommain}
Let $\tau$ and $\omega$ be arbitrary permutations over $[n]$.
Let $X$ and $Y$ be random graphs over the vertex set $[n]$, independently drawn from $\mathcal{G}(n,p)$ and $\mathcal{G}(n,q)$ respectively, where $p$ and $q$ satisfy $pq\ge 100\log n/n$. Then, with probability at least $1-o(n^{-2})$, the distance between $\tau$ and $\omega$ in $\FS(X,Y)$ is $O(n^5)$.
\end{restatable}

Alon, Defant and Kravitz studied the symmetric case where $p=q$. 
In order to prove connectivity of $\FS(X,Y)$, they required $p> p_0$ where $p_0^2 =\exp(4(\log n)^{2/3})/n$. They conjectured that a threshold of $p^2=\Theta(n^{-1})$ should suffice for connectivity. Milojevi\'c \cite{milojevic2022connectivity} disproved this conjecture, showing that $p^2=\Omega(\log n/n)$ is necessary for connectivity. Wang and Chen~\cite{wang2023connectivity} extended Alon et al.'s result to the asymmetric setting where the graphs $X$ and $Y$ have different edge probabilities $p$ and $q$.
Their result also requires $pq\ge p_0^2$.
In contrast, for our result a slightly smaller bound of $pq>O(1)\cdot \exp(\log\log n)/n$ suffices. This is because we bound the distance between a single pair of permutations with high probability. In particular, our result does not necessarily imply that $\FS(X,Y)$ is connected or has low diameter. We leave open the question of studying the diameter of $\FS(\mathcal{G}(n,p),\mathcal{G}(n,q))$.

Several of our results are shown in \cref{table:our-results}. 

\begin{table}[h!]
\centering
\setlength{\tabcolsep}{10pt}
\begin{tabular}{lll}
\toprule
\textbf{Families} & \textbf{Upper bound on} & \textbf{Lower bound on}\\
 & \textbf{max possible diameter} & \textbf{max possible diameter}\\
\midrule
General & $e^{O(n\log n)}$ \hfill (trivial) & $e^{\Omega(n)}$ \hfill \cite{jeong2022diameters} \\[1ex]
$\FS(K_n, Y)$ & $2n^2 - 5n + 3$ \cite{jeong2022diameters}, \textcolor{blue}{$\binom{n}{2}$} \hfill \cref{sec:specialcases} & \textcolor{blue}{$\binom{n}{2}$} \hfill \cref{sec:specialcases}\\[1ex]
$\FS(\Cycle_n, Y)$ & $8n^4(1 + o(1))$ \hfill \cite{jeong2022diameters} & $\binom{n}{2}$ \hfill \cite{jeong2022diameters}\\[1ex]
$\FS(\Path_n, Y)$ & $\binom{n}{2}$ \hfill \cite{jeong2022diameters} & $\binom{n}{2}$ \hfill \cite{jeong2022diameters}\\[1ex]
$\FS(\Star_n, \Tree_n)$ & \textcolor{blue}{$O(n)$} \hfill \cref{sec:specialcases} & \textcolor{blue}{$\Omega(n)$} \hfill \cref{sec:specialcases}\\[1ex]
$\FS(X, Y)$ where $x > y$ \\and $2x+y > 2n$ & \textcolor{blue}{$O(n^5)$} & Unknown\\[1ex]
$\FS(X, Y)$ where $x > y$ \\and $x+y > 3n/2$ & \textcolor{blue}{$O(n^2)$} & Unknown\\[1ex]
\bottomrule
\end{tabular}
\\[1em]
\caption{Diameter bounds for friends-and-strangers graphs in prior work (black) and our work (blue). These bounds refer to the largest possible diameter for any connected component of the graph. We refer to the minimum degree of $X$ as $x$ and the minimum degree of $Y$ as $y$. }
\label{table:our-results}
\setlength{\tabcolsep}{6pt}
\end{table}

\subsection{Organization of the paper}
In \cref{sec:background}, we present some notation and basic claims which we will use throughout the rest of the paper. Then, in \cref{sec:specialcases}, we bound the diameter of connected components of $\FS(X, Y)$ for special cases of $X$. In \cref{sec:bangachev}, we prove \cref{theorem:mindegreemain,theorem:mindegreesecond}, and in \cref{sec:random}, we prove \cref{theorem:randommain}. Finally, in \cref{sec:openquestions}, we discuss some open directions for future work. 
\section{Preliminaries}\label{sec:background}

\subsection{Notation}

In this subsection, we introduce the terminology which we use throughout this paper. For a graph~$G$, we denote the vertex set of $G$ as $V(G)$ and the edge set of $G$ as $E(G)$. 

In this paper, we will explore several families of graphs.
Here are the most important ones with vertex set $[n]=\{1,\dots,n\}$: 
\begin{itemize}
    \item the \emph{complete graph}, $K_n$, has an edge between any two vertices;
    \item the \emph{path graph}, $\Path_n$, has an edge between $i$ and $j$ when $j = i + 1$;
    \item the \emph{cycle graph}, $\Cycle_n$, has an edge between $i$ and $j$ when $j = i + 1 \bmod{n}$;
    \item the \emph{star graph}, $\Star_n$, has an edge between $i$ and $j$ when $i$ or $j$ is equal to $1$. 
\end{itemize}
We also define some other terms related to graph theory.
\begin{itemize}
    \item A \emph{cut vertex} in a graph is a vertex such that when removed, the graph gains at least one connected component. A \emph{separable} graph has at least one cut vertex. A graph is \emph{biconnected} if it is connected but not separable. 
    \item A graph on $n$ vertices has a \emph{Hamiltonian cycle} if it has a subgraph that is isomorphic to $\Cycle_n$. 
    \item A graph is \emph{$k$-regular} if each vertex has degree $k$. A graph is regular if there exists some $k$ for which it is $k$-regular. 
    \item The \emph{diameter} of a graph is the maximum distance between two vertices in a graph, where the \emph{distance} between two vertices in a graph is defined as the length of the shortest path between the two vertices. A graph's diameter is infinite if it is disconnected. 
    \item
    The minimum degree of a graph $G$, denoted as $\delta(G)$, is the infimum of the degrees of the vertices in $G$. 
\end{itemize}

We now introduce notation related to friends-and-strangers graphs. By identifying $V(X)$ and $V(Y)$, we refer to a vertex of $\FS(X,Y)$ as a {\em permutation} mapping the vertices of $X$ onto the locations in $Y$.
For a permutation $\sigma\in\FS(X,Y)$ and vertex $a\in V(X)$, we use $\sigma(a)\in V(Y)$ to denote the vertex in $Y$ that is located at $a$ in $\sigma$; Likewise, for a vertex $y\in V(Y)$, $\sigma^{-1}(y)$ denotes its location in $X$ under $\sigma$.

A move from a vertex to an adjacent vertex in $\FS(X,Y)$ is called an $(X,Y)$-{\em friendly swap} or, when $X$ and $Y$ are clear from the context, simply a {\em friendly swap} and mathematically denoted as a transposition. In particular, suppose $ab$ is an edge in $Y$ and $\sigma(a)\sigma(b)$ is an edge in $X$, then we represent the permutation resulting from the friendly swap of the elements $\sigma(a)$ and $\sigma(b)$ as $\sigma\circ(ab)$.

\subsection{Preliminary observations and past results}
Here we collect some general results on graphs and permutations that we will use throughout the paper. 

\begin{lemma}\label{lem:hi}
Any connected graph $G$ has a vertex which can be removed (along with its incident edges) such that the remaining graph is still connected. 
\end{lemma}
\begin{proof}

As $G$ is connected, it has some spanning tree $T$, which has some leaf $v$. Removing $v$, we see $T \setminus {v}$ is still a tree, in particular a spanning tree of $G \setminus {V}$, so $G \setminus {V}$ is still connected, as desired. 
\end{proof}

The following result is the aforementioned $\FS(\Star_n, Y)$ result from \cite{715921}. 

\begin{theorem}\label{theorem:kornhauser}
Let $Y$ be a graph on $n$ vertices. Then, the diameter of any connected component of $\FS(\Star_n, Y)$ is $O(n^3)$. 
\end{theorem}

This result forms the basis of several of our results later in this paper. We next present Wilson's result on the connectivity of $\FS(\Star_n, Y)$. 

\begin{theorem}[\cite{wilson1974graph}]\label{lem:wilson}
Let $Y$ be a biconnected graph on $n$ vertices which is not isomorphic to $\Cycle_n$ or $\theta(1,2,2)$. Then, if $Y$ is bipartite, $\FS(\Star_n, Y)$ has exactly $2$ connected components. Vertices are placed in these connected components based on the parity of the permutation without $s$. Otherwise, $\FS(\Star_n, Y)$ is connected. 
\end{theorem}
\section{Special Cases of Friends-and-Strangers Graphs}\label{sec:specialcases}
\subsection{$\FS(K_n, Y)$ for general $Y$}

Recall from \cref{table:our-results} that Jeong~\cite{jeong2022diameters} proved an upper bound of $2n^2 - 5n + 3$ for the diameter of connected components of $\FS(K_n, Y)$. We improve this bound to $\binom{n}{2}$ and show this is tight. 

\begin{theorem}\label{theorem:kn}
Let $Y$ be any graph on $n$ vertices. $\FS(K_n, Y)$ is connected if and only if $Y$ is connected. Furthermore, every connected component of $\FS(K_n, Y)$ has diameter less than or equal to $\binom{n}{2}$. Finally, $\FS(K_n, \Path_n)$ has a single connected component with diameter $\binom{n}{2}$. 
\end{theorem}

\begin{proof}
We will first assume that $Y$ is connected. We will show that $\FS(K_n, Y)$ is connected and its diameter is bounded by $\binom{n}{2}$. 
Let $\sigma$ be our current permutation, and $\tau$ our target permutation. For any vertex $v$ in $Y$, we will define $f(v) = \sigma(\tau^{-1}(v))$. This is the current location of the vertex in $K_n$ that we want to move to $v$.
Let $S\subset V(Y)$ be initialized as the empty set. Consider the following procedure, where we will maintain the invariant that the induced subgraph on $V(Y)\setminus S$ is connected:
\begin{enumerate}
    \item Let $v$ be a vertex in $V(Y)\setminus S$ such that the induced subgraph on $V(Y)\setminus (S \cup \{v\})$ is connected, which always exists by \cref{lem:hi}.
    \item Then, there exists a path in the induced subgraph on $V(Y) \setminus S$ from $v$ to $f(v)$. Make $(K_n, Y)$-friendly swaps until the person currently on $f(v)$ has moved to $v$.
    \item Add $v$ to $S$, and repeat this procedure until $S = V(Y)$.
\end{enumerate}
Note that at the end of this procedure, all people are in their intended positions, as once the person at $f(v)$ has moved to $v$, no friendly swaps are made involving $v$ for the rest of the procedure as henceforth $v\in S$.
Finally, each iteration of our procedure takes a maximum of $n - \abs{S} - 1$ friendly swaps, meaning that the total number of friendly swaps made is at most $(n-1) + (n-2) + \cdots = \binom{n}{2}$.

If $Y$ is not connected, for a permutation $\sigma$ and a connected component $C$ in $Y$, let $\sigma^{-1}(C)$ denote the set of vertices of $K_n$ located on $C$. For two permutations $\sigma$ and $\tau$, we claim that $\sigma$ and $\tau$ are connected if and only if for all connected components $C$ in $Y$, $\sigma^{-1}(C)=\tau^{-1}(C)$. On the one hand, if the latter condition does not hold, then there exists a vertex $v$ in $K_n$ that is located at different components in $Y$ under $\sigma$ and $\tau$. But no friendly move can move the vertex from one component to another, so $\sigma$ and $\tau$ cannot be connected. On the other hand, if the condition holds, then we can transform $\sigma$ to $\tau$ by transforming the permutation over each connected component in sequence. The number of steps taken, when $Y$ has connected components $C_1, \ldots, C_k$, is $\sum_{i\le k} \binom{\abs{C_i}}{2} \le \binom{n}{2}$.

We will now show that equality holds in \cref{theorem:kn} when $Y$ is isomorphic to $\Path_n$. 
We will number the vertices in $K_n$ $1, \cdots, n$, and number the vertices in $\Path_n$ $1, \cdots n$. 
Let $\sigma$ be the permutation $n, n-1, \cdots, 1$ of the vertices of $K_n$ onto $\Path_n$, and let $\tau$ be the permutation $1, 2, \ldots, n$. 
Define an \emph{inversion} in a permutation $\omega$ to be a pair $(i, j)$ of vertices in $K_n$ such that $i > j$ but $\omega(i) < \omega(j)$.

Observe that the number of inversions in $\sigma$ is $\binom{n}{2}$. 
Each $(K_n, \Path_n)$-friendly swap decreases the number of inversions of our current permutation by at most $1$. Hence, if the people start in the configuration $\sigma$, it will take them at least $\binom{n}{2}$ friendly swaps to reach the configuration $\tau$, which has $0$ inversions. 
\end{proof}

\subsection{$\FS(\Star_n, Y)$ when $Y$ is a forest} The main result of this section is that when $Y$ is a forest, every connected component in $\FS(\Star_n, Y)$ has linear diameter. We begin by proving this statement for trees.

\begin{lemma}\label{lem:trees}
Let $T$ be a tree on $n$ vertices. Then every connected component of $\FS(\Star_n, T)$ is isomorphic to $T$. 
\end{lemma}

\begin{proof}
Let $s$ be the center vertex in $\Star_n$. Fix a starting configuration with $s$ on vertex $v$, and let the connected component of $\FS(\Star_n, T)$ containing this permutation be $C$. We claim that there are precisely $n$ vertices in $C$, each of which has $s$ at a distinct vertex of $T$. 

Suppose that we move $s$ to another vertex $w$ through some sequence of friendly swaps. Observe that if at any point we move $s$ from some vertex $a$ to $b$ and then back from $b$ to $a$, the configuration has not changed (we have essentially done and undone an action). Hence, if $s$ is moved to $w$ in a non-simple path, then we can delete some friendly swaps from this sequence of moves. Hence, we may assume that $s$ is moved to $w$ in a simple path. Because there is exactly one simple path from $v$ to $w$ in $T$, we conclude that there is exactly one configuration with $s$ at $w$ in $C$. This shows that the number of vertices in $C$ is $n$.

To finish the claim, observe that two configurations in $C$ are adjacent if and only if the positions of $s$ in these configurations are adjacent. Hence, it follows that $C$ is isomorphic to $T$, with a natural bijection $T\to C$ mapping any vertex $v$ in $T$ to the position in $C$ with $s$ at $v$.
\end{proof}

We obtain the bound for forests as a simple corollary.
\begin{theorem}\label{theorem:forest}
Let $F$ be a forest on $n$ vertices. Then, the connected components of $\FS(\Star_n, F)$ have diameter at most $n - 1$. 
\end{theorem}
\begin{proof}
Again letting $s$ denote the central vertex in $\Star_n$, consider an arbitrary starting position $\sigma\in V(\FS(\Star_n,F))$, viewed as a permutation of people from $\Star_n$ onto positions from $F$.
Observe that if $C$ is a connected component of $F$ which $s$ is not in, then the people on $C$ cannot be moved.
Hence, if $C'$ is the connected component which $s$ is placed in, then by \cref{lem:trees}, the connected component containing $\sigma$ must be isomorphic to $C'$, so has diameter at most $\abs{C'}-1\leq n-1$.
\end{proof}

\section{General Diameter Bounds Depending on Minimum Degree}\label{sec:bangachev}

In \cite{bangachev2022asymmetric}, several conditions on the minimum degrees of $X$ and $Y$ are given to guarantee the connectedness of $\FS(X, Y)$. In this section, we will prove two similar theorems, but we will focus on bounding the diameter of our friends-and-strangers graphs rather than proving their connectivity. Our first result is based on a result of Bangachev \cite{bangachev2022asymmetric}. 

\begin{theorem}[\cite{bangachev2022asymmetric}]
Let $X$ and $Y$ be connected graphs on $n$ vertices such that \begin{align*}
\min(\delta(X), \delta(Y)) + 2\max(\delta(X), \delta(Y)) \geq 2n.
\end{align*}
Then, $\FS(X, Y)$ is connected. 
\end{theorem}

We improve on this result by showing that under the same conditions established in this theorem, the diameter of $\FS(X, Y)$ is polynomially bounded. 

\mindegreemain*

\begin{proof}
For a graph $G$, we denote the induced subgraph over a subset $S\subset V(G)$ as $G\mid_S$. Furthermore, for a vertex $i$ in $G$, define $N[i]$ as the closed neighborhood of $i$, or in other words the union of $\{i\}$ and the set of vertices in $G$ which are adjacent to $i$. 

Let $\sigma$ denote a bijection from $V(X)$ onto $V(Y)$, and let $u$ and $v$ be vertices in $Y$, with $\inv u \coloneq\sigma(u)$ and $\inv v \coloneq\sigma(v)$, such that $\inv u$ and $\inv v$ are adjacent in $X$. Let $Q$ be $\sigma(N[\sigma^{-1}(u)])$. In \cite{bangachev2022asymmetric}, it is shown that $\delta(Y|_Q) > \frac{1}{2}|Q|$.

We split our proof into two parts: first we will show that $\sigma$ and $\sigma \circ (\inv u, \inv v)$ lie in the same connected component of $\FS(X, Y)$ at distance at most $O(n^3)$. We refer to a sequence of friendly swaps from $\sigma$ to $\sigma \circ (\inv u, \inv v)$ as a $(u, v)$-exchange. Then, we will show that we need at most $\binom{n}{2}$ exchanges to go from any permutation $\sigma$ to any other permutation $\tau$ in $\FS(X, Y)$. These two facts complete the proof. 

For the first fact, we cite another theorem in Bangachev, which says that $\FS(\Star_n, Y)$ is connected if $\delta(Y) > \frac{n}{2}$. Hence, by \cref{theorem:kornhauser} we conclude that the diameter of $\FS(\Star_{|Q|}, Y|_Q)$ is at most $O(|Q|^3) \leq O(n^3)$.
Furthermore, $X|_{\sigma(Q)}$ clearly has a subgraph isomorphic to $\Star_{|Q|}$, because $u$ is adjacent to every vertex in this graph. Hence, we conclude that $\FS(\Star_{|Q|}, Y|_Q)$ is a subgraph of $\FS(X|_{\sigma(Q)}, Y|_Q)$, and the latter graph has diameter $O(n^3)$. The assertion follows, because this graph is connected and we can therefore clearly switch the positions of $\inv u$ and $\inv v$ in $O(n^3)$ moves. 

For the second assertion, define a graph $H$ where the vertices are permutations over $V(X)$. Two permutations $\sigma$ and $\sigma'$ are adjacent in $H$ iff $\sigma'=\sigma\circ (\inv u, \inv v)$ for some $\inv u, \inv v\in V(X)$ where $\sigma(\inv u)$ and $\sigma(\inv v)$ are adjacent in $Y$.
Then, as we showed in the first part, any two permutations $\sigma$ and $\sigma'$ that are adjacent in $H$ lie at distance at most $O(n^3)$ in $\FS(X, Y)$. Furthermore, $H$ is clearly isomorphic to $\FS(X, K_n)$, which we have proven has diameter at most $\binom{n}{2}$. This completes the second part, and we are done. 
\end{proof}

Our second result provides a much stronger diameter bound than the first result, given a similar but stronger condition on $\delta(X)$ and $\delta(Y)$.

\mindegreesecond*

\begin{proof}
Let $\sigma$ denote the bijection from $V(X)$ onto $V(Y)$.
Let $p$ and $q$ be vertices in $X$ such that $pq$ is an edge and let $a=\sigma(p)$ and $b=\sigma(q)$. We will prove the existence of a vertex $r$ in $Y$ that satisfies the following properties: (i) $r$ is adjacent to $a$ in $Y$; (ii) $r$ is adjacent to $b$ in $Y$; (iii) $\sigma^{-1}(r)$ is adjacent to $p$ in $X$; and, (iv) $\sigma^{-1}(r)$ is adjacent to $q$ in $X$. 

At most $n-1-\delta(Y)$ fail the condition (i); at most $n-1-\delta(Y)$ fail (ii); at most $n-1-\delta(X)$ fail (iii); and at most $n-1-\delta(X)$ fail (iv). In all, at most $4n - 4 - 2\delta(X) - 2\delta(Y) < n - 2$ vertices fail to satisfy one or more of these conditions. Therefore, there exists at least one vertex other than $a$ and $b$ that satisfies all of these conditions. Call it $r$.

Because $r$ is adjacent to $a$ and $b$ and $\sigma^{-1}(r)$ is adjacent to $p$ and $q$, we can perform the sequence of friendly swaps $(r, p)$, $(p, q)$, $(r, q)$, after which the positions of $p$ and $q$ are swapped. Hence, for any vertices $p$ and $q$ in $X$ for which $pq$ is an edge, it is possible to swap the positions of $p$ and $q$ in $3$ friendly swaps. 

Finally, observe that the diameter of $\FS(X, K_n)$ is at most $\binom{n}{2}=\frac{n(n-1)}{2}$. Hence, if we take a sequence of moves from any configuration in $\FS(X, K_n)$ to any other configuration, and replace each move with it corresponding sequence of $3$ friendly swaps, we can move from any configuration to any other configuration in $\FS(X, Y)$ in at most $\frac{3n(n-1)}{2}$ friendly swaps. 
\end{proof}

We remark that Bangachev \cite{bangachev2022asymmetric} proves another connectivity result of a similar flavor as the connectivity analog of \cref{theorem:mindegreemain}; see \cref{fig:Banga} for a visual comparison of the regimes that these results apply to. We state this result below for completeness.

\begin{theorem}\protect{\cite[Theorem 1.4]{bangachev2022asymmetric}}
Let $X$ and $Y$ be graphs on $n$ vertices. If $\delta(X) > n/2$, $\delta(Y) > n/2$, and $2\min(\delta(X), \delta(Y)) + 3\max(\delta(X), \delta(Y)) \geq 3n$, then $\FS(X, Y)$ is connected. 
\end{theorem}

We leave open the question of finding the diameter analog of this result.
\section{Diameter Bounds for Erd\H{o}s-R\' enyi Random Graphs}\label{sec:random}

In this section we prove \cref{theorem:randommain}, which we restate below for convenience.
\randommain*

As in the work of Alon, Defant and Kravitz \cite{alon2021extremal}, our proof relies on the notion of Wilsonian graphs, as defined below. \cref{lem:wilson} then shows that if $Y$ is Wilsonian, then $\FS(\Star_n, Y)$ is connected. We will use this fact in our argument.

\begin{definition}
\label{def:wilson}
    We say a graph $G$ is {\em Wilsonian} if it is biconnected, non-bipartite, and neither a cycle graph with at least 4 vertices nor isomorphic to $\theta(1,2,2)$.
\end{definition}

We will also need the following lemma about biconnectivity in random graphs.
\begin{lemma}
\label{lem:random-bicon}
    Let $X\sim\mathcal{G}(n,p)$ with $p\ge 20\log n/n$. Then $X$ is biconnected with probability at least $1-o(n^{-8})$.
\end{lemma}

\begin{proof}
    As having higher $p$ increases the probability that $X$ is biconnected, without loss of generality assume $p=20\log n/n$.
    We will show that $X$ is biconnected by showing that $X$ does not contain a cut vertex with high probability. We will compute the probability that there exists a vertex $v$ which is a cut vertex. If $v$ is a cut vertex, let $M$ and $N$ be the vertex sets of the connected components of $X \setminus \{v\}$. Assume without loss of generality that $\abs{M} \leq \abs{N}$. 

    We will split this into two cases: \begin{enumerate}
        \item We first consider the case $\abs{M}\leq 20$.
        Fix any nonempty subset $M\subseteq V(X)$ of cardinality at most 20. Then for any vertex in $M$, the number of edges it has with the $n-\abs{M}\geq n-20$ vertices in $V(X)\setminus M$ is at most 1, namely an edge to $v$.
        We will use Chernoff bounds to bound the probability that any fixed $x\in M$ has at most 1 edge to $V(X)\setminus M$.
        Note that these events for different $x\in M$ are mutually independent.

        The expected number of edges among these $n-\abs{M}$ possible edges is
        \[ \mu=20\log n - \frac{20M\log n}{n} \geq 20\log n - \frac{400\log n}{n}. \]
        Therefore, if we define $\delta=\frac{\mu-1}{\mu}$, then the probability $x\in M$ has at most 1 edge to $V(X)\setminus M$ is at most $e^{-\delta^2\mu/2}$.
        For large enough $n$, we have $\delta^2 \geq 0.991$, so this value is bounded from above by \begin{align*}
            e^{-0.991\mu/2} < n^{-9.9}.
        \end{align*}
        
        To finish this case, note that our analysis does not assume $v$ is fixed; we have shown that for fixed $M$, the probability $X$ has a cut vertex that separates $X$ into two connected components, the smaller of which is $M$, is at most $n^{-9.9\abs{M}}$.
        In other words, this $n^{-9.9\abs{M}}$ bound includes all possible values of $v$.
        Thus, it only remains to take a union bound over all possibilities for $M$, so the probability of this case occurring is at most \begin{align*}
            \sum_{\abs{M} = 1}^{20} \binom{n}{\abs{M}} \cdot (n^{-9.9})^{\abs{M}}
            = o(n^{-8}).
        \end{align*}
        \item We will now tackle the case in which $\abs{M}$ is greater than $20$. In this case, there are no edges between $M$ and $N$. However, there are $\abs{M}(n - 1 - \abs{M})$ pairs of vertices in $M$ and $N$. Hence, the probability of this case occurring is at most \begin{align*}
            \sum_{r = 21}^{n/2 - 1} \binom{n}{r} \cdot (n - r) \cdot (1 - p)^{r(n - 1 - r)}.
        \end{align*}

        Here, the first term represents the number of ways to choose the subgraph $M$, the second term represents the number of ways to choose $v$, and the third term represents the probability that none of the edges between $M$ and $N$ exist. We can bound this from above by \begin{align*}
            \sum_{r = 21}^{n/2 - 1} n^{r+1}(1-p)^{r(n - 1 - r)} & 
            \leq \sum_{r = 21}^{n/2 - 1} \exp((r+1)\log(n) - pr(n - 1 - r))\\
            & \leq \sum_{r = 21}^{n/2 - 1} \exp((r+1)\log(n) - 10r\log(n))\\
            & \leq \sum_{r = 21}^{n/2 - 1} \exp(-8r\log(n))\\
            & = \sum_{r = 21}^{n/2 - 1} n^{-8r}
            = O(n^{-168}).
        \end{align*}
    \end{enumerate}

    Therefore, the probability that $X$ has a cut vertex is at most $o(n^{-8})$. It follows that the probability that $X$ is biconnected is $1 - o(n^{-8})$. 
\end{proof}

Now we are ready to prove the main result of this section.
\begin{proof}[Proof of \cref{theorem:randommain}]
    Let $X$ and $Y$ be random graphs chosen independently from $\mathcal{G}(n,p)$ and $\mathcal{G}(n,q)$, respectively, over the vertex set $[n]$. Without loss of generality we assume $p\ge q$, so in particular, $p\ge 10\sqrt{\log n/n}$. Let $m = (n-1)p$ and $\epsilon=\sqrt{10\log n/m}$. Observe that 
    $\epsilon=o(1)$. 
    
    We first note that, by \cref{lem:random-bicon}, 
    $X$ is connected with probability $1-o(n^{-8})$.
    Moreover, with high probability the degree of every vertex in $X$ is in $[(1- \epsilon)m,(1+\epsilon)m]$.
    In particular, for a fixed vertex $v\in V(X)$, the expected degree of $v$ in $X$ is $m$, and therefore using Chernoff bounds, the probability that $v$'s degree falls outside of the range $[(1- \epsilon)m,(1+\epsilon)m]$ is at most $\exp(-\epsilon^2 m/2) = n^{-5}$. Taking the union bound over the $n$ vertices, we get that with probability at least $1-n^{-4}$, every vertex in $X$ has degree in $[(1- \epsilon)m,(1+\epsilon)m]$ and $X$ is connected.
    We will condition on this event for the rest of the argument.

    Fix the graph $X$, any permutation $\sigma:V(X)\to V(Y)$, and any pair of vertices $a,b \in V(Y)$ such that $\inv a\coloneq\sigma^{-1}(a)$ and $\inv b\coloneq\sigma^{-1}(b)$ are adjacent in $X$. We will consider the event $\mathcal{E}_{\sigma,a,b}$ that there exists a sequence of $O(n^3)$ moves in $\FS(X,Y)$ that implement the transposition $(\inv a, \inv b)$. We will show that over the choice of $Y$ drawn from $\mathcal{G}(n,q)$, the probability that $\mathcal{E}_{\sigma,a,b}$ does not occur is $o(n^{-4})$.

    First, note that this bound implies the result.
    Let $\tau$ and $\omega$ be any two permutations in $\FS(X,Y)$, which is a subgraph of $\FS(X,K_n)$. Since $X$ is connected, so is $\FS(X,K_n)$, and by \cref{theorem:kn}, there is a path $P$ of length at most $\binom{n}{2}$ between $\tau$ and $\omega$ in $\FS(X,K_n)$. Each step in this path corresponds to a move from $\sigma$ to $\sigma\circ(\inv a,\inv b)$ for some $\sigma$, $a$, and $b$ with $\inv a$ and $\inv b$ being neighbors in $X$. If the event $\mathcal{E}_{\sigma,a,b}$ occurs, this step can be replicated in $\FS(X,Y)$ by a sequence of $O(n^3)$ moves. Then, taking the union bound over each step in $P$, we get that the distance between $\tau$ and $\omega$ is at most $O(n^3)\cdot \binom{n}{2} = O(n^5)$ with probability at least $1- \binom{n}{2} \cdot o(n^{-4}) = 1-o(n^{-2})$. The theorem then follows.
    
    It remains to prove the $1 - o(n^{-4})$ lower bound on the probability of occurrence of $\mathcal{E}_{\sigma,a,b}$. Recall that $\inv a=\sigma^{-1}(a)$ and $\inv b=\sigma^{-1}(b)$ are neighbors. Let $T_X$ denote the closed neighborhood of $\inv a$ in $X$. Note that $\abs{T_X}\in[(1-\epsilon)m+1,(1+\epsilon)m+1]$. Consider the set of locations of $T_X$ under $\sigma$ in $Y$, namely, $T_Y = \sigma(T_X)$. We have $a,b\in T_Y$. 
    
    Let $H_X$ and $H_Y$ denote the induced subgraphs $X|_{T_X}$ and $Y|_{T_Y}$ respectively. If $H_Y$ is Wilsonian, then by \cref{lem:wilson}, $\FS(\Star_{\abs{T_X}}, H_Y)$ is connected, and by \cref{theorem:kornhauser} it has diameter $O(\abs{T_X}^3)$, which is also $O(n^3)$. In this case, $\FS(H_X,H_Y)$, which contains $\FS(\Star_{\abs{T_X}}, H_Y)$ as a subgraph where the central vertex of $\Star_{\abs{T_X}}$ is $\inv a$, has diameter $O(n^3)$, and thus contains a path from $\sigma|_{T_Y}$ to $\sigma|_{T_Y}\circ(\inv a, \inv b)$ of length $O(n^3)$. In other words, if $H_Y$ is Wilsonian, then $\mathcal{E}_{\sigma,a,b}$ occurs.

    We will now bound the probability that $H_Y$ is Wilsonian. We do this by bounding the probability of the three possible ways $H_Y$ could not be Wilsonian, where we use $t$ to denote $\abs{T_Y}$.
    \begin{itemize}
        \item {\bf $H_Y$ is not biconnected.} Note that
        \[ q\ge \frac{100\log n}{np}>50\frac{\log n}{t}. \]
        Then, by \cref{lem:random-bicon}, it follows that $H_Y$ is not biconnected with probability at most $o(t^{-8})= o(n^{-8/2})$, where we used $t\ge (1-\epsilon)(n-1)p=\Omega(\sqrt{n\log n})$. Therefore, $H_Y$ is biconnected with probability $1 - o(n^{-4})$.
        \item {\bf $H_Y$ is bipartite.} We can bound the probability of this event by summing over all possible bipartitions the probability that the bipartition arises: 
        \begin{align*}
            \sum_{j=0}^t \binom{t}{j} (1-q)^{\binom{j}{2} + \binom{t-j}{2}} \le 2^t (1-q)^{t^2/4}< 2^te^{-qt^2/4}<\exp(t-t\log n)=o(n^{-4}).
        \end{align*}
        Note that we used the fact that $qt\ge (1-\epsilon)(n-1)pq>4\log n$ and $t>5$.
        \item {\bf $H_Y$ is a cycle.} For any fixed vertex in $H_Y$, its expected degree is $(t-1)q$. We can bound the probability that its degree is at most $2<(t-1)q/4$ using the Chernoff bound by $\exp(-(t-1)q/8)=o(n^{-4})$. The probability that $H_Y$ is a cycle is no more than the probability that one of its vertices has degree at most $2$.
    \end{itemize}
    Taking the union bound over the three bad events and applying \cref{def:wilson}, we get that the probability that $H_Y$ is not Wilsonian is at most $o(n^{-4})$, and so we are done.
\end{proof}

\section{Open Questions}\label{sec:openquestions}

In this section, we state several open questions which we hope to see addressed in future work. 
The first two questions were proposed by Jeong \cite{jeong2022diameters}. This first question is related to Jeong's theorem that showed that $\FS(X, Y)$ can have connected components of diameter $e^{\Omega(n)}$. 

\begin{question}[\cite{jeong2022diameters}]
Let $X$ and $Y$ be graphs on $n$ vertices. Is it true that the maximum possible diameter of a connected component of $\FS(X, Y)$ is $e^{O(n)}$?
\end{question}

Jeong's friends-and-strangers graph construction for the equality case in the above question has many connected components. This raises a natural follow-up question.

\begin{question}[\cite{jeong2022diameters}]
Is it true that if $\FS(X, Y)$ is connected, then its diameter is polynomially bounded in terms of $n$?
\end{question}

\cite{715921} shows that the diameter of the connected components of $\FS(\Star_n, Y)$ is bounded polynomially. We can ask a similar but more general question about trees, as follows.

\begin{question}
Let $T$ be a tree and $Y$ be an arbitrary graph, both with $n$ vertices. Is the diameter of any connected component of $\FS(T, Y)$ polynomially bounded in $n$?
\end{question}

We would also like to extend our results in \cref{sec:bangachev}, as shown in \cref{fig:Banga}.
We showed that the diameter of $\FS(X, Y)$ is polynomial in $n$ when $\min(\delta(X), \delta(Y)) + 2\max(\delta(X), \delta(Y)) \geq 2n$. On the other hand, Bangachev \cite{bangachev2022asymmetric} showed that $\FS(X, Y)$ is connected when $\delta(X), \delta(Y) > \frac{n}{2}$ and $2\min(\delta(X), \delta(Y)) + 3\max(\delta(X), \delta(Y)) \geq 3n$.
This result covers many graphs which are not shown to have polynomial diameter in \cref{theorem:mindegreemain}. We ask whether a counterpart of \cref{theorem:mindegreemain} extends to this regime. 

\begin{question}
Is it true that when $\delta(X), \delta(Y) > \frac{n}{2}$ and $2\min(\delta(X), \delta(Y)) + 3\max(\delta(X), \delta(Y)) \geq 3n$, $\diam(\FS(X, Y))$ is polynomially bounded in $n$?
\end{question}

Bangachev \cite{bangachev2022asymmetric} also conjectured that the condition $\delta(X), \delta(Y) > \frac{n}{2}$ is not necessary for connectivity. Therefore, we ask the following natural question.

\begin{question}
Is it true that when $2\min(\delta(X), \delta(Y)) + 3\max(\delta(X), \delta(Y)) \geq 3n$, $\diam(\FS(X, Y))$ is polynomially bounded in $n$?
\end{question}

In \cref{sec:random}, we showed that when $X$ and $Y$ are Erd\H os-R\' enyi random graphs satisfying certain  conditions, the distance between any two permutations in $\FS(X, Y)$ is $O(n^6)$ with high probability. We are curious to know if in the same regime, one can establish polynomial diameter bounds with high probability. 

\begin{question}
Let $p$ and $q$ be probabilities and $c$ a constant for which $pq \geq \frac{c\log{n}}{n}$. Is it true with high probability that $\FS(\mathcal{G}(n, p), \mathcal{G}(n, q))$ has connected components with $\poly(n)$ diameter?
\end{question}

There are many other questions about friends-and-strangers graphs which are still open; we refer the interested reader to Jeong \cite{jeong2022diameters} for a list with additional open questions.

\section*{Acknowledgments}
Rupert Li was partially supported by a Hertz Fellowship and a PD Soros Fellowship.
We are grateful to the PRIMES-USA program and its organizers, Dr.\ Tanya Khovanova, Prof.\ Pavel Etingof, and Dr.\ Slava Gerovitch, for providing us the opportunity to conduct this research. Furthermore, we thank Dr.\ Khovanova for proofreading our work and providing helpful feedback.

\bibliographystyle{amsinit}
\bibliography{reportref}

\end{document}